\documentclass{amsart}
\usepackage{amssymb,amsmath,mathtools}
\usepackage{stmaryrd}
\usepackage{xspace}
\usepackage{enumerate}
\usepackage{hyperref}
\usepackage{color}

\newcommand{\bX}{{\mathbf X}}

\newcommand{\DIV}{\nabla\!{\cdot}}   
\newcommand{\ROT}{\nabla\!\times\!}  
\newcommand{\GRAD}{\nabla}           
\newcommand{\LAP}{{\Delta}}          
\newcommand{\ADV}{{\cdot}\nabla} 
\newcommand{\Real}{\mathbb R}

\newcommand{\vare}{{\varepsilon}}
\newcommand{\dt}{{\tau}}

\newcommand{\ie}{i.e.,\@\xspace}

\newcommand{\cf}{cf.\@\xspace}

\newcommand{\pe}{\textup{\textsf{p}}}
\newcommand{\ue}{\textup{\textsf{u}}}
\newcommand{\we}{\textup{\textsf{w}}}
\def\Hun{{{    H}^{1}(\Omega)}}
\def\Hunz{{{   H}^1_0 (\Omega)}}
\newcommand{\frakm}{{\mathfrak m}}
\newcommand{\fraki}{{\mathfrak i}}
\newcommand{\polQ}{{\mathbb Q}}
\newcommand{\calC}{{\mathcal C}}
\newcommand{\hunzd}{{\mathbf{H}_0^1}}
\newcommand{\hunz}{{H_0^1}}
\newcommand{\bL}{{\mathbf L}}
\def\scl{\left\langle}
\def\scr{\right\rangle}
\def\tildeLdeux{{L^2_{\scriptscriptstyle\!\int\!=0} (\Omega)}}
\def\Ldeux{{{  L}^2   (\Omega)}}
\newcommand{\polN}{{\mathbb N}}
\newcommand{\frakd}{{\mathfrak d}}
\def\Hunzd{{{  \bf H}^1_0 (\Omega)}}
\newcommand{\bH}{{\mathbf H}}
\def\Hdeuxd{{{   \bf H}^2   (\Omega)}}
\newcommand{\bW}{{\mathbf W}}
\newcommand{\calE}{{\mathcal E}}
\newcommand{\calO}{{\mathcal O}}

\numberwithin{equation}{section}
\theoremstyle{plain}
\newtheorem{theorem}[equation]{Theorem}
\newtheorem{lemma}[equation]{Lemma}
\newtheorem{corollary}[equation]{Corollary}
\newtheorem{proposition}[equation]{Proposition}

\theoremstyle{definition}
\newtheorem{remark}[equation]{Remark}

\newtheorem*{acknowledgement*}{Acknowledgement}


\begin{document}

\title[Fluids with Microstructure]{Convergence analysis of a fractional time-stepping technique for incompressible fluids with microstructure}

\author[A.J.~Salgado]{Abner J.~Salgado}
\address{Department of Mathematics, University of Tennessee, Knoxville, TN 37996, USA.}
\email{asalgad1@utk.edu}


\date{\textcolor{blue}{Submitted \today}}


\begin{abstract}
We present and analyze a fully discrete fractional time stepping technique for the solution of the micropolar Navier Stokes equations, which is a system of equations that describes the evolution of an incompressible fluid whose material particles possess both translational and rotational degrees of freedom. The proposed scheme uncouples the computation of the linear and angular velocity and the pressure. It is unconditionally stable and delivers optimal convergence rates.
\end{abstract}

\maketitle

\section{Introduction}
\label{sec:Intro}
One of the fundamental assumptions in fluid mechanics is that material particles do not possess angular momentum and that there are no distributed couples. As a consequence of this we obtain that the stress tensor is symmetric. However, this approach is not satisfactory in the case when the orientability of the material particles is important for the process of interest. This is the case, for instance, when dealing with anisotropic or polarizable fluids and so the conservation of angular momentum must be taken explicitly into account to describe the behavior of such fluid.

This work is concerned with the development of a fractional time stepping technique for the so-called micropolar Navier Stokes equations, which read
\begin{equation}
\label{eq:mupnse}
  \begin{dcases}
    \ue_t + \ue\ADV \ue -(\nu + \nu_r) \LAP \ue + \GRAD \pe = 2\nu_r \ROT \we + f, \\
    \DIV \ue = 0, \\
    \mathfrak{j} \we + \mathfrak{j} \ue \ADV \we - (c_a + c_d) \LAP \we - (c_0 + c_d - c_a) \GRAD \DIV \we + 4 \nu_r \we = 2 \nu_r \ROT \ue + g,
  \end{dcases}
\end{equation}
where $\ue$ is the linear velocity of the fluid, $\pe$ the pressure and $\we$ is the angular velocity. All the material constants $\mathfrak{j}$, $\nu$, $\nu_r$, $c_a$, $c_d$ and $c_0$ are assumed to be constant, positive and satisfy
\[
  c_0 + c_d - c_a > 0.
\]
The terms $f$ and $g$ represent a smooth, externally applied, force and moment, respectively. The fluid occupies a bounded domain $\Omega \subset \Real^d$ with $d=2,3$ and a solution to the above problem is sought over the time interval $[0,T]$. We supplement the system \eqref{eq:mupnse} with initial and boundary conditions for the linear and angular velocities:
\begin{equation}
\label{eq:IBCs}
  \begin{dcases}
    \ue|_{t=0} = \ue_0, & \ue|_{\partial\Omega} = 0, \\
    \we|_{t=0} = \we_0, & \we|_{\partial\Omega} = 0.
  \end{dcases}
\end{equation}

System \eqref{eq:mupnse}--\eqref{eq:IBCs} describes, under certain constitutive assumptions \cite{NST,Luka}, the evolution of an incompressible fluid in which the material particles possess rotational degrees of freedom. The origin of this system can be traced back to \cite{Dahler1,Dahler2} and \cite{EringMicro,Eringen99,Eringen01}, where a general theory of continua with microstructure is developed. The interested reader is also referred to \cite{Luka} for an analysis of this model. Let us also mention that this system is a component in the so-called Rosensweig model for ferrofluids \cite{Ros97}, where the system is coupled to the (stationary) Maxwell equations, and an equation for the so-called magnetization field through the forcing terms $f$ and $g$ which depend on the magnetic field and magnetization. An analysis of the Rosensweig model of ferrohydrodynamics can be found in \cite{MR2430804,Ami10}. We finally mention that potential applications of ferrofluids, and thus of the micropolar Navier Stokes equations, are discussed in \cite{NST}.

Despite the fact that, as mentioned in the previous paragraph, system \eqref{eq:mupnse} has a great deal of practical applications, to the best of our knowledge only two references deal with its discretization: \cite{Ort08} proposes and analyzes a penalty projection-method and suboptimal error estimates are proved.
Reference \cite{NST}, develops a semi-implicit fully discrete scheme which, at each time step, decoupled the computation of the linear and angular velocities but required the solution of a saddle-point problem for the determination of the linear velocity and pressure. In other words, if we knew $(\ue^l,\pe^l,\we^l)$ for all $l=\overline{0,k}$ we found first $(\ue^{k+1},\pe^{k+1})$ by taking the angular velocity explicit and solving the corresponding saddle point problem (first two lines of \eqref{eq:mupnse}). Once the linear velocity $\ue^{k+1}$ is known we can solve the convection--diffusion equation (last line in \eqref{eq:mupnse}) that determines $\we^{k+1}$. The authors of this work show that this method is unconditionally stable and delivers optimal error estimates. Its main bottleneck, however, was the solution of the saddle point problem. In this work we remedy this shortcoming by applying a fractional time stepping technique.

Fractional time stepping techniques date back to the late 1960's and the work of Chorin \cite{Chor68,Chor69} and Temam \cite{MR0237972,tema69} and are among the most popular methods for the solution of the incompressible Navier Stokes equations. A detailed exposition on this techniques can be found in \cite{MR2250931}. Simply put these methods are based on a realization of the Helmholtz decomposition of the velocity into a solenoidal field and a gradient. Another, more recent, point of view on fractional time stepping techniques has emerged \cite{Guermond20092834,MR2802553} by which these schemes are understood as a penalization on the divergence in a negative norm. This is the point of view that we shall adopt and, in this sense, our work will be a combination of the techniques and ideas of \cite{NST} and \cite{MR2802553}.

This work is organized as follows. The notation, as well as the particulars of the time and space discretization are described in Section~\ref{sec:prelims}. Section~\ref{sec:Scheme} presents the scheme that we will be concerned with as well as its stability analysis. On the basis of the stability estimates the error analysis is shown in Section~\ref{sec:Errors}, where we obtain optimal error estimates, both in time and space, for the linear and angular velocities, as well as for the pressure. Finally, to illustrate the performance of the method numerical experiments are shown in Section~\ref{sec:Numerics}.

\section{Notation and Preliminaries}
\label{sec:prelims}

We will consider system \eqref{eq:mupnse}--\eqref{eq:IBCs} on the finite time interval $[0,T]$ and in an open, connected and bounded domain $\Omega \subset \Real^d$ ($d=2,3$) which, for simplicity, we assume is sufficiently smooth. Moreover we assume that \eqref{eq:mupnse}--\eqref{eq:IBCs} has a unique and sufficiently smooth solution, see \S~\ref{sub:vels} and \S~\ref{sub:preserr} for a precise statement. To simplify notation we will set
\[
  \nu_0 = \nu + \nu_r, \quad c_1 = c_a + c_d, \quad c_2 = c_0 + c_d - c_a.
\]

Whenever $X$ is a normed space, we denote by $\|\cdot\|_X$ its norm and by $X'$ its dual. We say that a time dependent function $\phi : [0,T] \rightarrow X$ is in $L^p(X)$ whenever the map $(0,T) \ni t \mapsto \|\phi(t)\|_X^p \in \Real$ is integrable. No notational distinction is made for vector valued functions but their spaces will be denoted by bold fonts. We use the standard Sobolev spaces $W^k_p(\Omega)$ consisting of functions whose distributional derivatives of order up to $k$ are in $L^p(\Omega)$. As usual, we set $\Hun = W^1_2(\Omega)$ and $\Hunz$ as the closure of $\calC^\infty_0(\Omega)$ in $\Hun$ which we norm with $\|v \|_{\hunz} = \| \GRAD v\|_{\bL^2}$. We denote the space of $\Ldeux$ functions with vanishing mean value by $\tildeLdeux$. The inner product in $\Ldeux$ is denoted by $\scl q, r \scr = \int_\Omega qr$.

In what follows the relation $ A \lesssim B$ means that $A \leq c B$ for a constant that might depend on the data of the problem and its exact solution, but it does not depend on the discretization parameters nor the solution of the discrete scheme. The value of this constant might change at each occurrence.

\subsection{Time Discretization}
\label{sub:time}
We discretize the time interval $[0,T]$ by introducing a number of steps $K \in \polN$ so that the time-step is $\dt = T/K$ and we set $t_k = k \tau$ for $k=\overline{0,K}$. We denote $\phi^k := \phi(t_k)$ and $\phi^\dt = \{ \phi^k \}_{k=0}^K$. The time increment operator $\frakd$ is defined by setting
\begin{equation}
\label{eq:defoffrakd}
  \frakd \phi^k = \phi^k - \phi^{k-1},
\end{equation}
with $\frakd^2\phi^k = \frakd(\frakd \phi^k) = \phi^k - 2\phi^{k-1} + \phi^{k-2}$, and we define the discrete norms
\begin{equation}
\label{eq:l2linf}
  \| \phi^\dt \|_{\ell^2(X)}^2 = \dt \sum_{k=0}^K \| \phi^k \|_X^2, \qquad
  \| \phi^\dt \|_{\ell^\infty(X)} = \max\left\{ \| \phi^k \|_X : k=\overline{0,K} \right\}.
\end{equation}

\subsection{Space Discretization}
\label{sub:space}
We will approximate the solution to \eqref{eq:mupnse}--\eqref{eq:IBCs} with Galerkin techniques. We introduce two families of finite dimensional spaces $\{\bX_h\}_{h>0}$ and $\{M_h\}_{h>0}$ with $\bX_h \subset \Hunzd$ and $M_h \subset \tildeLdeux$. The spaces $\bX_h$ will be used to approximate the linear and angular velocity, while the pressure will be approximated in $M_h$. The pair $(\bX_h,M_h)$ must be compatible, in the sense that they satisfy the well known LBB condition 
\begin{equation}
\label{eq:LBB}
  \|q_h \|_{L^2} \lesssim \sup \left\{ \frac{ \int_\Omega q_h \DIV v_h }{ \|v_h\|_{\hunzd} } : 0 \neq v_h \in \bX_h \right\} , \quad \forall q_h \in M_h.
\end{equation}
In addition, these spaces possess suitable approximation properties, \ie there is $\frakm \in \polN$ such that for all $m \in [0,\frakm]$ we have
\begin{equation}
\label{eq:Xhapprox}
  \inf\left\{ \| v - v_h \|_{\bL^2} + h \| v - v_h \|_{\hunzd} : v_h \in \bX_h  \right\} \lesssim h^{m+1} \| v \|_{\bH^{m+1}},
\end{equation}
for all $v \in \bH^{m+1}(\Omega) \cap \Hunzd$, and
\begin{equation}
\label{eq:Mhapprox}
  \inf\left\{ \| q - q_h \|_{\bL^2}  : q_h \in M_h  \right\} \lesssim h^{m} \| q \|_{H^{m}}, \quad
  \forall q \in H^{m}(\Omega) \cap \tildeLdeux.
\end{equation}
Finally, we assume that the velocity space satisfies the following inverse inequality
\begin{equation}
\label{eq:inverseineq}
  \| v_h \|_{\bL^\infty} \lesssim \fraki(h) \|v_h \|_{\hunzd}, \quad \forall v_h \in \bX_h, \qquad
  \fraki(h) = \begin{dcases}
              \left( 1 + |\log(h)| \right)^{1/2}, & d = 2, \\
              h^{-1/2}, & d = 3.
            \end{dcases}
\end{equation}
In the context of finite elements, examples of spaces satisfying these properties can be found in \cite{GR86,MR2050138}.

Following the ideas of \cite{MR0351124,MR2249024}, for $t \in (0,T]$, we define the Stokes projection of $(\ue(t),\pe(t))$ as the pair $(\ue_h(t),\pe_h(t)) \in \bX_h \times M_h$ such that, for every $(v_h , q_h) \in \bX_h \times M_h$, satisfies
\begin{equation}
\label{eq:defStokesproj}
  \begin{dcases}
    \nu_0 \scl \GRAD \ue_h(t), \GRAD v_h \scr + \scl \GRAD \pe_h(t), v_h \scr = \nu_0 \scl \GRAD \ue(t), \GRAD v_h \scr - \scl \pe(t), \DIV v_h \scr,
    \\
    \scl q_h, \DIV \ue_h(t) \scr = 0,
  \end{dcases}
\end{equation}
and the elliptic-like projection of $\we(t)$ as the function $\we_h(t) \in \bX_h$ that solves
\begin{multline}
\label{eq:defofEllproj}
  c_1 \scl \GRAD \we_h(t), \GRAD z_h \scr + c_2 \scl \DIV \we_h(t), \DIV z_h \scr + 4\nu_r \scl \we_h(t), z_h \scr = 
  c_1 \scl \GRAD \we(t), \GRAD z_h \scr + \\
  c_2 \scl \DIV \we(t), \DIV z_h \scr + 4\nu_r \scl \we(t), z_h \scr,
  \quad \forall z_h \in \bX_h.
\end{multline}
We assume that these projectors satisfy:

\begin{lemma}[Properties of projectors]
\label{lem:projs}
Assume that $(\ue,\we,\pe) \in L^\infty(\Hdeuxd\cap\Hunzd)^2 \times L^\infty(\Hun\cap\tildeLdeux)$, then the Stokes and elliptic projectors are stable, \ie
\[
  \| \ue_h \|_{L^\infty(\bL^\infty \cap \bW^1_3)} + \| \we_h \|_{L^\infty(\bL^\infty \cap \bW^1_3)} + 
  \| \pe_h \|_{L^\infty(H^1)} \lesssim 1.
\]
If, in addition, $(\ue,\we,\pe) \in L^\infty(\bH^{\frakm+1}(\Omega)\cap\Hunzd)^2 \times L^\infty(H^{\frakm}(\Omega)\cap\tildeLdeux)$ then the projectors have the following approximation properties
\begin{align*}
  \| \ue - \ue_h \|_{L^\infty(\bL^2)} + h \| \ue - \ue_h \|_{L^2(\hunzd)} + h\| \pe - \pe_h \|_{L^\infty(L^2)} &\lesssim h^{\frakm+1}, \\
  \| \we - \we_h \|_{L^\infty(\bL^2)} + h \| \we - \we_h \|_{L^2(\hunzd)} &\lesssim h^{\frakm+1}.
\end{align*}
\end{lemma}

Finally, we introduce the discrete trilinear form \cite{Temam}
\[
  b_h(u,v,w) = \scl (u \ADV v) , w \scr + \frac12 \scl (\DIV u) v , w \scr, \quad \forall u,v,w \in \Hunzd,
\]
and recall that it is consistent: $b_h(u,v,w) = \scl (u \ADV v) , w \scr$ whenever $\DIV u = 0$ and skew-symmetric, \ie
\[
  b(u,v,v)=0, \quad \forall u,v \in \Hunzd.
\]
The following are well-known bounds for this form:
\begin{align}
\label{eq:1stbineq}
  b_h(u_h, v_h w_h ) &\lesssim \| u_h \|_{\hunzd} \| v_h \|_{\hunzd} \| w_h \|_{\hunzd}, \\
\label{eq:2ndbineq}
  b_h(u_h,v_h,w_h) &\lesssim \fraki(h) \| u_h \|_{\bL^2} \| v_h \|_{\hunzd} \| w_h \|_{\hunzd},
\end{align}
for all $u_h,v_h,w_h \in \bX_h$.

\begin{remark}[$d=2$]
There is a slight difference between \eqref{eq:mupnse} in two and three dimensions. Namely, if $d=2$ the field $\we$ is a (pseudo)scalar whereas if $d=3$ is a (pseudo)vector. This inconsistency must be taken into account in the choice of discrete spaces for the angular velocity. In what follows we will develop the scheme and carry out the analysis under the assumption that $d=3$. In the two dimensional case the angular velocity must be taken from a discrete (scalar valued) space $X_h$ and all the analysis presented below follows after minor modifications.
\end{remark}

\section{Description of the Scheme}
\label{sec:Scheme}
Let us now describe the scheme. As mentioned in Section~\ref{sec:Intro}, our method will decouple the linear and angular velocities as in \cite{NST} and the linear velocity and pressure with a fractional time stepping technique as in \cite{MR2802553}. We compute three sequences $u_h^\tau$, $w_h^\tau$ and $p_h^\tau$ which approximate, respectively, the linear and angular velocities and the pressure. 

Given the initial data $(\ue_0,\we_0)$ we construct $(u_h^0,w_h^0,p_h^0) \in \bX_h^2 \times M_h$ that approximate, respectively, the initial linear and angular velocities and pressure. We assume that these approximations satisfy
\begin{align*}
  \| \ue_0 - u_h^0 \|_{\bL^2} + h \| \ue_0 - u_h^0 \|_{\hunzd} &\lesssim h^{\frakm+1}, \\
  \| \we_0 - w_h^0 \|_{\bL^2} + h \| \we_0 - w_h^0 \|_{\hunzd} &\lesssim h^{\frakm+1}, \\
  \| \pe_0 - p_h^0 \|_{L^2} &\lesssim h^{\frakm+1}.
\end{align*}
After initialization, for $k=\overline{0,K-1}$, we advance in time in several stages:
\begin{enumerate}[$\bullet$]
  \item \textbf{Pressure extrapolation:} Define
  \begin{equation}
  \label{eq:psharp}
    p_h^\sharp = \begin{dcases}
                   p_h^k, & k=0, \\
                   2 p_h^k - p_h^{k-1}, & k >0.
                 \end{dcases}
  \end{equation}
  
  \item \textbf{Linear velocity update:} Find $u_h^{k+1} \in \bX_h$ that solves
  \begin{multline}
  \label{eq:velupdate}
    \scl \frac{\frakd u_h^{k+1}}\dt, v_h \scr + b_h(u_h^k, u_h^{k+1}, v_h) + \nu_0 \scl \GRAD u_h^{k+1}, \GRAD v_h \scr + \scl \GRAD p_h^\sharp, v_h \scr 
    = \\
    2 \nu_r \scl \ROT w_h^k, v_h \scr + \scl f^{k+1}, v_h \scr - \scl \GRAD \frakd \psi_h^{k+1}, v_h \scr,
    \quad v_h \in \bX_h
  \end{multline}

  \item \textbf{Pressure update:} Find $p_h^{k+1} \in M_h$ that solves
  \begin{equation}
  \label{eq:presupdate}
    \scl \GRAD \frakd p_h^{k+1}, \GRAD r_h \scr = \frac1\dt \scl u_h^{k+1}, \GRAD r_h \scr + \scl \GRAD \psi_h^{k+1}, \GRAD r_h \scr,
    \quad \forall r_h \in M_h.
  \end{equation}

  \item \textbf{Angular velocity update:} Find $w_h^{k+1} \in \bX_h$ that solves
  \begin{multline}
  \label{eq:wupdate}
    \mathfrak{j} \scl \frac{\frakd w_h^{k+1}}\dt, z_h \scr + \mathfrak{j} b_h (u_h^{k+1}, w_h^{k+1}, z_h ) + c_1 \scl \GRAD w_h^{k+1}, \GRAD z_h \scr
    + c_2 \scl \DIV w_h^{k+1}, \DIV z_h \scr \\
    + 4\nu_r \scl w_h^{k+1}, z_h \scr =
    2\nu_r \scl \ROT u_h^{k+1}, z_h \scr + \scl g^{k+1}, z_h \scr, \quad \forall z_h \in \bX_h.
  \end{multline}
\end{enumerate}

\begin{remark}[Definition of $\psi_h^\dt$]
In equations \eqref{eq:velupdate} and \eqref{eq:presupdate} that define our scheme we introduced the variable $\psi_h^\dt$. This variable must be set to zero, \ie $\psi_h^\dt \equiv 0$. We have included this term to shorten the discussion and analysis, as it will be seen below.
\end{remark}

\begin{remark}[Initialization]
Notice that $\pe_0 = \pe|_{t=0}$ is not part of the initial data of our problem. Under suitable compatibility and smoothness assumptions on the initial data and forcing terms this quantity can be computed by solving a Poisson equation. One can also assume, for instance, that $k=0$ does not correspond to $t=0$ but that a few time steps with a coupled scheme (like the one presented in \cite{NST}) have been computed. This will also simplify \eqref{eq:psharp}. This shortcoming is not particular to our scheme, but rather a recurring feature for fractional time stepping techniques. If these conditions are not met the analysis of the scheme must be adapted to account for weighted (in time) error estimates as it is detailed in \cite{MR650052,MR1472237}.
\end{remark}

Let us record the stability properties of this scheme.

\begin{proposition}[A priori estimate on the velocities]
\label{prop:apriori}
Let $(u_h^\dt,w_h^\dt,p_h^\dt) \subset \bX_h^2 \times M_h$ solve \eqref{eq:psharp}---\eqref{eq:wupdate}. Then, the following a priori estimate holds:
\begin{multline}
\label{eq:apriori}
  \frakd\left( \| u_h^{k+1}\|_{\bL^2}^2 + \mathfrak{j} \|w_h^{k+1}\|_{\bL^2}^2 \right) + \mathfrak{j} \|\frakd w_h^{k+1} \|_{\bL^2}^2 \\
  + 2\dt \left( \nu \| u_h^{k+1} \|_{\hunzd}^2 + c_1 \| w_h^{k+1} \|_{\hunzd}^2 + c_2 \| \DIV w_h^{k+1} \|_{\bL^2}^2  \right) \\
  + \dt^2 \left( \frakd \| \GRAD p_h^{k+1} \|_{\bL^2} + \| \frakd \GRAD p_h^k \|_{\bL^2}^2 \right) \leq
  2\dt \left( \scl f^{k+1}, u_h^{k+1} \scr  +  \scl g^{k+1}, w_h^{k+1} \scr \right. \\
  \left. 
  - \scl \GRAD \frakd \psi_h^{k+1}, u_h^k \scr + \dt \scl \GRAD \psi_h^{k+1}, \GRAD p_h^\sharp \scr \right).
\end{multline}
\end{proposition}
\begin{proof}
The proof combines the ideas of \cite[Proposition~3.1]{NST} and \cite[Theorem~3.1]{Guermond20092834}. Set $v_h = 2\dt u_h^{k+1}$ in \eqref{eq:velupdate} and 
$z_h = 2 \dt w_h^{k+1}$ in \eqref{eq:wupdate} and add the results. We obtain
\begin{multline}
\label{eq:stab1}
  \frakd\left( \| u_h^{k+1}\|_{\bL^2}^2 + \mathfrak{j} \|w_h^{k+1}\|_{\bL^2}^2 \right) + \| \frakd u_h^{k+1}\|_{\bL^2} + \mathfrak{j} \|\frakd w_h^{k+1} \|_{\bL^2}^2 \\
  + 2\dt \left( \nu \| u_h^{k+1} \|_{\hunzd}^2 + c_1 \| w_h^{k+1} \|_{\hunzd}^2 + c_2 \| \DIV w_h^{k+1} \|_{\bL^2}^2  \right) +
  2\dt \scl \GRAD p_h^\sharp, u_h^{k+1} \scr \leq \\
  2 \dt \left( \scl f^{k+1}, u_h^{k+1} \scr  +  \scl g^{k+1}, w_h^{k+1} \scr
  - \scl \GRAD \frakd \psi_h^{k+1}, u_h^{k+1} \scr \right),
\end{multline}
where we repeated the arguments used in \cite[Proposition~3.1]{NST}. It remains to obtain a bound for the pressure term and to obtain it, we follow \cite[Theorem~3.1]{Guermond20092834}. We begin by noticing that $p_h^\sharp = p_h^{k+1} - \frakd^2 p_h^{k+1}$ so that, setting $r_h = 2 \dt^2 p_h^\sharp$ in \eqref{eq:presupdate} yields
\begin{multline}
\label{eq:stab2}
  \dt^2 \left( \frakd \| \GRAD p_h^{k+1} \|_{\bL^2} + \| \frakd \GRAD p_h^k \|_{\bL^2}^2 - \| \frakd^2 \GRAD p_h^{k+1} \|_{\bL^2}^2 \right) \\
  = 2 \dt \scl \GRAD p_h^\sharp, u_h^{k+1} \scr + 2 \dt^2 \scl \GRAD \psi_h^{k+1}, \GRAD p_h^\sharp \scr.
\end{multline}
Apply the operator $\frakd$ to \eqref{eq:presupdate} and set $r_h = 2\dt \frakd^2 p_h^{k+1}$. Using the Cauchy-Schwarz inequality we obtain
\begin{equation}
\label{eq:stab3}
  \dt^2 \| \frakd^2 \GRAD p_h^{k+1} \|_{\bL^2}^2 \leq \| \frakd u_h^{k+1} \|_{\bL^2}^2 + \dt^2 \| \frakd \GRAD \psi_h^{k+1} \|_{\bL^2}^2 
  + 2\dt \scl \frakd u_h^{k+1}, \GRAD \frakd \psi_h^{k+1} \scr.
\end{equation}
Adding \eqref{eq:stab1}--\eqref{eq:stab3} we obtain the result.
\end{proof}

Notice that this is indeed an \emph{a priori} estimate, since the right hand side of \eqref{eq:apriori} depends only on the data of the problem and the solution at previous time steps. When dealing with our scheme, \ie $\psi_h^\dt \equiv 0$, this gives us a stability estimate.

\begin{corollary}[Stability]
\label{col:stability}
Let $(u_h^\dt,w_h^\dt,p_h^\dt) \subset \bX_h^2 \times M_h$ solve \eqref{eq:psharp}---\eqref{eq:wupdate} with $\psi_h^\dt \equiv 0$. Then it satisfies the 
following stability estimate
\[
  \| u_h^\dt \|_{\ell^\infty(\bL^2)} + \| w_h^\dt \|_{\ell^\infty(\bL^2)} + \| u_h^\dt \|_{\ell^2(\hunzd)} + \| w_h^\dt \|_{\ell^2(\hunzd)}
  \lesssim
  \| f^\dt \|_{\ell^2(\bH^{-1})} + \| g^\dt \|_{\ell^2(\bH^{-1})}.
\]
\end{corollary}
\begin{proof}
Set $\psi_h^\dt \equiv 0$ in \eqref{eq:apriori}, use the Cauchy-Schwarz inequality and add over $k$.
\end{proof}

In Corollary~\ref{col:stability} we did not obtain stability estimates for the pressure. This is so because in \eqref{eq:apriori} the terms that involve the pressure are multiplied by a factor $\dt^2$ and so they do not scale properly, moreover these are of the form $\|\GRAD p_h^{k+1}\|_{\bL^2}$ which is not the natural norm for the pressure. An estimate for the pressure must be obtained using the LBB condition \eqref{eq:LBB}. This is the content of the following result.

\begin{proposition}[A priori estimate on the pressure]
\label{prop:aprpress}
Assume that $(u_h^\dt, w_h^\dt, p_h^\dt) \subset \bX_h^2 \times M_h$ solves \eqref{eq:psharp}--\eqref{eq:wupdate}. Then we have
\begin{multline}
\label{eq:presest}
  \| p_h^{\sharp,\dt} \|_{\ell^2(L^2)}^2 \lesssim \frac1\dt \sum_{k=1}^K \| \frakd u_h^k \|_{\bL^2}^2 + \| u_h^\dt \|_{\ell^2(\hunzd)}^2
  + \| w_h^\dt \|_{\ell^2(\hunzd)}^2 + \| f^\dt \|_{\ell^2(\bH^{-1})}^2 \\
  + \| \frakd \psi_h^\dt \|_{\ell^2(L^2)}^2
  + \dt \fraki(h)^2 \sum_{k=0}^{K-1} \| u_h^k \|_{\bL^2}^2 \| u_h^{k+1} \|_{\hunzd}^2.
\end{multline}
\end{proposition}
\begin{proof}
Owing to the LBB condition \eqref{eq:LBB} from \eqref{eq:velupdate} we obtain 
\begin{multline*}
  \| p_h^\sharp \|_{L^2} \lesssim \frac1\dt \| \frakd u_h^{k+1} \|_{\bL^2} + \nu_0 \| u_h^{k+1} \|_{\hunzd} + 2\nu_r \| w_h^{k+1} \|_{\bL^2} +
  \| f^{k+1} \|_{\bH^{-1}} + \| \frakd \psi_h^{k+1} \|_{L^2}^2 \\
  + \sup\left\{ \frac{ b_h (u_h^k, u_h^{k+1}, v_h ) }{ \| v_h \|_{\hunzd} } : 0 \neq v_h \in \bX_h \right\}.
\end{multline*}
Owing to \eqref{eq:2ndbineq}, we have
\[
  \sup\left\{ \frac{ b_h (u_h^k, u_h^{k+1}, v_h ) }{ \| v_h \|_{\hunzd} } : 0 \neq v_h \in \bX_h \right\} \lesssim
  \fraki(h) \| u_h^{k+1} \|_{\bL^2} \| u_h^{k+1} \|_{\hunzd}.
\]
Insert this estimate on the previous inequality, square it and multiply it by $\dt$ to obtain
\begin{multline*}
  \dt \| p_h^\sharp \|_{L^2}^2 \lesssim \frac1\dt \| \frakd u_h^{k+1} \|_{\bL^2}^2 + \dt \| u_h^{k+1} \|_{\hunzd}^2 + \dt \| w_h^{k+1} \|_{\bL^2}^2
  + \dt \|f^{k+1} \|_{\bH^{-1}}^2 + \dt \| \frakd \psi_h^{k+1} \|_{L^2}^2 \\
  + \dt \fraki(h)^2 \| u_h^k \|_{\bL^2}^2 \| u_h^{k+1} \|_{\hunzd}^2.
\end{multline*}
Adding over $k=\overline{0,K-1}$ we obtain the result.
\end{proof}

The conclusion of Proposition~\ref{prop:aprpress} gives an a priori estimate on the pressure provided
\begin{equation}
\label{eq:apriori_deriv}
  \| \frakd u_h^\dt \|_{\ell^\infty(\bL^2)} \lesssim \dt^2
\end{equation}
holds. Indeed, in this case the right hand side is bounded. The use of such an estimate, \ie how to obtain \eqref{eq:apriori_deriv} shall become clear once we perform the error analysis.

\section{Error Estimates}
\label{sec:Errors}

Let us now carry out the error analysis of scheme \eqref{eq:psharp}--\eqref{eq:wupdate}. We will do so, as it is accustomed, by identifying the equations that the errors satisfy. As it turns out, these quantities solve \eqref{eq:psharp}--\eqref{eq:wupdate} for properly chosen $f^\dt$, $g^\dt$ and $\psi_h^\dt$. Thus, for the linear and angular velocities, the a priori estimate provided in Proposition~\ref{prop:apriori} reduces the analysis to finding suitable estimates for the right hand sides, which are formed by consistency terms. An error estimate on the pressure, however, requires a bound of the form \eqref{eq:apriori_deriv} which we must first derive. This will require to work with increments of the errors. Once this estimate is obtained, we can apply Proposition~\ref{prop:aprpress} to conclude.

In order to provide error estimates we will assume that
\begin{equation}
\label{eq:smoothness}
  \ue,\we \in W^2_\infty\left( \Hunzd \cap \bH^{\frakm+1}(\Omega) \right), \qquad \pe \in W^2_\infty \left(\tildeLdeux\cap H^\frakm(\Omega) \right).
\end{equation}

\subsection{Consistency Analysis}
\label{sub:consistency}
Using the projectors defined in \S\ref{sub:space} we introduce
\begin{align*}
  \eta_\ue^\dt &= \ue^\dt - \ue_h^\dt, & \eta_\we^\dt &= \we^\dt - \we_h^\dt, & \eta_\pe^\dt &= \pe^\dt - \pe_h^\dt, \\
  E_h^\dt &= \ue_h^\dt - u_h^\dt, & \calE_h^\dt &= \we_h^\dt - w_h^\dt, & \vare_h^\dt &= \pe_h^\dt - p_h^\dt.
\end{align*}
The sequences $\eta_\ue^\dt$, $\eta_\we^\dt$ and $\eta_\pe^\dt$ are called the interpolation errors, whereas $E_h^\dt$, $\calE_h^\dt$ and $\vare_h^\dt$ are termed the approximation errors. Owing to the properties of projectors stated in Lemma~\ref{lem:projs} to obtain an error estimate it suffices to bound the approximation errors, which is what we concentrate on below.

Take the difference of the first equation in \eqref{eq:defStokesproj} and \eqref{eq:velupdate}. Do the same for the second equation in \eqref{eq:defStokesproj} and \eqref{eq:presupdate}; and for \eqref{eq:defofEllproj} with \eqref{eq:wupdate}. Proceeding this way we obtain that the interpolation errors $(E_h^\dt, \calE_h^\dt, \vare_h^\dt) \subset \bX_h^2 \times M_h$ satisfy \eqref{eq:psharp}--\eqref{eq:wupdate} with
\begin{equation}
\label{eq:defofconserrors}
  \left\{
  \begin{aligned}
  \scl f^{k+1}, v_h \scr &= \scl \frac{ \frakd \ue_h^{k+1} }\dt - [\ue_t]^{k+1}, v_h \scr  + b_h(E_h^k, E_h^{k+1}, v_h) \\
          & - b_h(u_h^k, u_h^{k+1}, v_h )
          - b_h( \ue^{k+1}, \ue^{k+1}, v_h ) \\
          & + 2\nu_r \scl \ROT(\we^{k+1} - \we_h^k), v_h \scr, \\
  \scl g^{k+1}, z_h \scr &= \mathfrak{j} \scl \frac{\frakd \we_h^{k+1} }\dt - [\we_t]^{k+1}, z_h \scr
      + \mathfrak{j} b_h(E_h^{k+1}, \calE_h^{k+1}, z_h ) \\
  &+ \mathfrak{j} b_h(u_h^{k+1}, w_h^{k+1}, z_h )
  - \mathfrak{j} b_h(\ue^{k+1}, \we^{k+1}, z_h ) \\
  & + 2\nu_r \scl \ROT(\ue^{k+1} - \ue_h^{k+1}), z_h \scr , \\
  \psi_h^k & = \frakd \pe_h^k.
  \end{aligned}
  \right.
\end{equation}

\subsection{Error Estimates on the Velocities}
\label{sub:vels}

As the analysis in \S~\ref{sub:consistency} shows, the approximation errors solve \eqref{eq:psharp}--\eqref{eq:wupdate} with right hand sides given by \eqref{eq:defofconserrors}. Consequently, thanks to Proposition~\ref{prop:apriori}, a bound on these terms will allow us to provide error estimates for the linear and angular velocities.

\begin{theorem}[Error estimates on $u_h^\dt$ and $w_h^\dt$]
\label{thm:erruw}
The solution to \eqref{eq:psharp}--\eqref{eq:wupdate} satisfies
\[
  \| E_h^\dt \|_{\ell^\infty(\bL^2)} + \| E_h^\dt \|_{\ell^2(\hunzd)} + 
  \| \calE_h^\dt \|_{\ell^\infty(\bL^2)} + \| \calE_h^\dt \|_{\ell^2(\hunzd)} \lesssim \dt + h^{\frakm+1}.
\]
\end{theorem}
\begin{proof}
From \eqref{eq:apriori}, we only need to prove a bound on the consistency terms \eqref{eq:defofconserrors}.

\noindent \framebox{\textbf{Bounds on $f$:}} Owing to the regularity of the Stokes and elliptic-like projectors stated in Lemma~\ref{lem:projs} we readily obtain that
\[
  \left\| \frac{ \frakd \ue_h^{k+1} }\dt - [\ue_t]^{k+1} \right \|_{\bL^2} +
  \left\| \we^{k+1} - \we_h^k \right\|_{\bL^2}
  \lesssim \dt + h^{\frakm+1}.
\]
In addition, it is rather standard (\cf \cite{GuQu_98,NST,MR2802553}) to obtain that
\begin{multline*}
  b_h(E_h^k, E_h^{k+1}, v_h) - b_h(u_h^k, u_h^{k+1}, v_h ) \\
  - b_h( \ue^{k+1}, \ue^{k+1}, v_h )
  \lesssim
  \left( \dt + h^{\frakm + 1} + \|E_h^k \| \| E_h^{k+1} \|_{\hunzd} \right) \|v_h \|_{\hunzd}.
\end{multline*}

\noindent \framebox{\textbf{Bounds on $g$:}} Similarly we obtain
\[
  \left\| \frac{ \frakd \we_h^{k+1} }\dt - [\we_t]^{k+1} \right \|_{\bL^2} +
  \left\| \ue^{k+1} - \ue_h^{k+1} \right\|_{\bL^2}
  \lesssim \dt + h^{\frakm+1}.
\]
The estimates of \cite[Theorem 4.1]{NST} yield
\begin{multline*}
  b_h(E_h^{k+1}, \calE_h^{k+1}, v_h) - b_h(u_h^{k+1}, w_h^{k+1}, v_h ) - b_h( \ue^{k+1}, \we^{k+1}, v_h )
  \lesssim \\
  \left( h^{\frakm + 1} \| \calE_h^{k+1} \|_{\hunzd} + \|E_h^k \| \| \calE_h^{k+1} \|_{\hunzd} \right) \|v_h \|_{\hunzd}.
\end{multline*}

\noindent \framebox{\textbf{Bounds on $\psi$:}} Again, thanks to Lemma~\ref{lem:projs}
\[
  2\dt \scl \GRAD \frakd^2 \pe_h^{k+1}, E_h^k \scr \leq \frac\dt2 \|E_h^k\|_{\bL^2}^2 + c\dt^5.
\]
Finally,
\begin{align*}
  2\dt^2 \scl \GRAD \frakd \pe_h^{k+1}, \GRAD \vare_h^\sharp \scr &= 
  2\dt^2 \scl \GRAD \frakd \pe_h^{k+1}, \GRAD \vare_h^k \scr + 2\dt^2 \scl \GRAD \frakd \pe_h^{k+1}, \GRAD \frakd\vare_h^k \scr \\
  &\leq \dt^3 \| \GRAD \vare_h^k \|_{\bL^2}^2 + \dt \| \GRAD \frakd \pe_h^{k+1} \|_{\bL^2}^2 + \dt^2 
  + \dt^2 \| \GRAD \frakd \pe_h^{k+1} \|_{\bL^2}^2 \\
  &+ \dt^2 \| \GRAD \frakd \vare_h^k \|_{\bL^2}^2 
   \leq \dt^3 \| \GRAD \vare_h^k \|_{\bL^2}^2 + \dt^2 \| \GRAD \frakd \vare_h^k \|_{\bL^2}^2 + c \dt^3.
\end{align*}

An application of Gr\"onwall's inequality allows us to conclude.
\end{proof}

The ability of $\dt^{-1} \frakd u_h^{k+1}$ to approximate the derivative, as well as an instance of \eqref{eq:apriori_deriv} is given in the following.

\begin{proposition}[Estimates on the increments]
\label{prop:inc}
Assume that there exists $h_\fraki > 0 $ such that for every $h \in (0,h_\fraki]$ we have $\dt \fraki(h) \lesssim 1$. Then, for $h \in (0,h_\fraki]$, the solution to \eqref{eq:psharp}---\eqref{eq:wupdate} satisfies
\[
  \| \frakd E_h^\dt \|_{\ell^\infty(\bL^2)} + \| \frakd \calE_h^\dt \|_{\ell^\infty(\bL^2)} \lesssim \dt ( \dt + h^\frakm ).
\]
\end{proposition}
\begin{proof}
The proof is technical and tedious but rather standard. One proceeds by taking the difference of two consecutive time steps of \eqref{eq:psharp}---\eqref{eq:wupdate} to find the equations that control the increments $\frakd E_h^{k+1}$, $\frakd \calE_h^{k+1}$ and $\frakd \vare_h^{k+1}$. They turn out to be, again, of the form \eqref{eq:psharp}---\eqref{eq:wupdate} so that, by Proposition~\ref{prop:apriori}, one only needs to bound the right hand sides. It turns out that all of them are of the right order. Let us only remark that the restriction on the time step is necessary because one needs to use \eqref{eq:2ndbineq} to control expressions containing the trilinear form. The reader is referred to \cite{GuQu_98,NST} for such a type of estimate.
\end{proof}

\subsection{Error Estimates on the Pressure}
\label{sub:preserr}

The estimate provided in Proposition~\ref{prop:inc} is an analogue of \eqref{eq:apriori_deriv} and provides the key step in obtaining estimates on the pressure.

\begin{theorem}[Error estimates on $p_h^\dt$]
\label{thm:pres}
Assume that there exists $h_\fraki > 0 $ such that for every $h \in (0,h_\fraki]$ we have $\dt \fraki(h) \lesssim 1$. Then, for $h\in (0,h_\fraki]$, the solution to \eqref{eq:psharp}---\eqref{eq:wupdate} satisfies
\[
  \| \vare_h^{\sharp, \dt} \|_{\ell^2(L^2)} \lesssim \dt + h^{\frakm}.
\]
\end{theorem}
\begin{proof}
From Proposition~\ref{prop:aprpress} it suffices to bound the right hand side of \eqref{eq:presest}. From Theorem~\ref{thm:erruw} we have
\[
  \| E_h^\dt \|_{\ell^2(\hunzd)} + \| \calE_h^\dt \|_{\ell^2(\hunzd)} + \| f^\dt \|_{\ell^2(\bH^{-1})}
  + \| \frakd^2 \pe_h^\dt \|_{\ell^2(L^2)} \lesssim \dt + h^{\frakm + 1}.
\]
The assumptions allow us to conclude, using Proposition~\ref{prop:inc}, that
\[
  \frac1\dt \sum_{k=1}^K \| \frakd E_h^k \|_{\bL^2}^2 \lesssim (\dt + h^\frakm)^2.
\]
Finally,
\[
  \dt \fraki(h)^2 \sum_{k=0}^{K-1} \| E_h^k \|_{\bL^2}^2 \| E_h^{k+1} \|_{\hunzd}^2
  \lesssim \dt^3 \fraki(h)^2 \sum_{k=0}^{K-1} \| E_h^{k+1} \|_{\hunzd}^2
  \lesssim \| E_h^\dt \|_{\ell^2(\hunzd)}^2 \lesssim (\dt + h^{\frakm+1})^2,
\]
where we used Theorem~\ref{thm:erruw} and the assumption on the relation between $\dt$ and $h$.
\end{proof}

\section{Numerical Experiments}
\label{sec:Numerics}

To illustrate the performance of the method we have developed and analyzed in the previous sections here we present a series of numerical experiments. The implementation has been carried out with the help of the \texttt{deal.II} library, \cite{BHK2007,DealIIReference}. For the discretization of the linear velocity and pressure we use the lowest order Taylor--Hood element $\polQ_2/\polQ_1$ and for the discretization of the angular velocity we use continuous $\polQ_2$ elements. In this case then $\frakm = 2$.

We set
\[
  \mathfrak{j} = \nu = \nu_r = c_0 = c_a = c_d = 1,
\]
and solved \eqref{eq:mupnse}--\eqref{eq:IBCs} on $\Omega = (-1,1)^2 \subset \Real^2$ with right hand sides $f$ and $g$ chosen so that the exact solution is
\begin{align*}
  \ue(x,y,t) &= \pi \sin(t) \left( \sin^2(\pi x) \sin(2\pi y), \ - \sin(2 \pi x) \sin^2(\pi y) \right)^\intercal, \\
  \pe(x,y,t) &= \sin(t) \cos(\pi x) \sin( \pi y), \\
  \we(x,y,t) &= \pi \sin(t) \sin^2( \pi x ) \sin^2( \pi y ).
\end{align*}

\begin{table}
\begin{center}
  \begin{tabular}{||c||c|c||c|c||}
  \hline
  $\dt$ & $\| \ue - u_h^\dt \|_{\ell^\infty(\bL^2)}$ & Rate &  $\|\ue - u_h^\dt\|_{\ell^2(\hunzd)}$ & Rate \\
  \hline
  1.0000e-01  & 4.8106e-02  & ---   & 6.3890e-01  & ---  \\
  \hline
  5.0000e-02  & 1.7379e-02  & 1.47  & 2.5639e-01  & 1.32 \\
  \hline
  2.5000e-02  & 6.9459e-03  & 1.32  & 1.0962e-01  & 1.23 \\
  \hline
  1.2500e-02  & 3.2290e-03  & 1.11  & 5.1982e-02  & 1.08 \\
  \hline
  6.2500e-03  & 1.5824e-03  & 1.03  & 2.5924e-02  & 1.00 \\
  \hline
  3.1250e-03  & 7.8735e-04  & 1.01  & 1.3717e-02  & 0.92 \\
  \hline
  1.5625e-03  & 3.9323e-04  & 1.00  & 8.3217e-03  & 0.72 \\
  \hline
  \end{tabular}
\end{center}
\caption{Errors in time for the linear velocity}
\label{tab:errtimeU}
\end{table}

\begin{table}
\begin{center}
  \begin{tabular}{||c||c|c||}
  \hline
  $\dt$ & $\| \pe - p_h^\dt \|_{\ell^2(L^2)}$ & Rate \\
  \hline
  1.0000e-01  & 1.0542e+00  & ---   \\
  \hline
  5.0000e-02  & 4.5970e-01  & 1.20  \\
  \hline
  2.5000e-02  & 2.0780e-01  & 1.15  \\
  \hline
  1.2500e-02  & 9.7081e-02  & 1.10  \\
  \hline
  6.2500e-03  & 4.7005e-02  & 1.05 \\
  \hline
  3.1250e-03  & 2.3201e-02  & 1.02 \\
  \hline
  1.5625e-03  & 1.1537e-02  & 1.01 \\
  \hline
  \end{tabular}
\end{center}
\caption{Errors in time for the pressure}
\label{tab:errtimeP}
\end{table}

\begin{table}
\begin{center}
  \begin{tabular}{||c||c|c||c|c||}
  \hline
  $\dt$ & $\| \we - w_h^\dt  \|_{\ell^\infty(\bL^2)}$ & Rate  & $\| \we - w_h^\dt \|_{\ell^2(\hunzd)}$ & Rate \\
  \hline
  1.0000e-01  & 9.3011e-03  & ---   & 7.2865e-02  & ---   \\
  \hline
  5.0000e-02  & 3.7720e-03  & 1.30  & 3.3060e-02  & 1.14  \\
  \hline
  2.5000e-02  & 1.7508e-03  & 1.11  & 1.6211e-02  & 1.03  \\
  \hline
  1.2500e-02  & 8.7158e-04  & 1.01  & 8.3214e-03  & 0.96  \\
  \hline
  6.2500e-03  & 4.3623e-04  & 1.00  & 4.6360e-03  & 0.84  \\
  \hline
  3.1250e-03  & 2.1819e-04  & 1.00  & 3.0964e-03  & 0.58  \\
  \hline
  1.5625e-03  & 1.0976e-04  & 0.99  & 2.5717e-03  & 0.27  \\
  \hline
  \end{tabular}
\end{center}
\caption{Errors in time for the angular velocity}
\label{tab:errtimeW}
\end{table}

The space approximation properties of scheme \eqref{eq:psharp}--\eqref{eq:wupdate} are like those of the scheme presented in \cite{NST}, where numerical experiments also presented. For this reason here we concentrate on the temporal accuracy.

To illustrate the accuracy in time of the developed scheme we consider a mesh consisting of $65536$ quadrilateral cells. The dimensions of the discrete spaces are as follows
\[
  \dim X_h = 526338, \qquad \dim \bX_h = 263169, \qquad \dim M_h = 66049.
\]
This way the space discretization error is negligible in comparison with the time discretization error. We set $T=10$ and varied $\dt$ in the range $10^{-3}\leq \dt \leq 10^{-1}$. Tables~\ref{tab:errtimeU}--\ref{tab:errtimeW} show the results for the linear velocity, pressure and angular velocity, respectively. As wee see, all the errors are of order $\calO(\dt)$.

\bibliographystyle{plain}
\bibliography{biblio}

\end{document}